\documentclass{amsart}

\usepackage[all]{xy}

\theoremstyle{plain}
\newtheorem{thm}{Theorem}[section]
\newtheorem{prop}[thm]{Proposition}
\newtheorem{lemma}[thm]{Lemma}

\theoremstyle{definition}
\newtheorem{dfn}[thm]{Definition}

\newtheorem{hypo}[thm]{Hypothesis}

\theoremstyle{remark}

\newcommand{\D}{\mathrm{D}}
\newcommand{\SD}{\mathrm{SD}}
\newcommand{\Q}{\mathrm{Q}}

\begin{document}

\title{Universal deformation rings and tame blocks}

\author{Frauke M. Bleher}
\address{F.B.: Department of Mathematics\\University of Iowa\\
Iowa City, IA 52242-1419, U.S.A.}
\email{frauke-bleher@uiowa.edu}
\thanks{The first author was supported in part by  
NSA Grant H98230-11-1-0131.}
\author{Giovanna LLosent}
\address{G.L.: Department of Mathematics\\CSU
San Bernardino, CA 92407-2397, U.S.A.}
\email{gllosent@csusb.edu}
\author{Jennifer B. Schaefer}
\address{J.S.: Department of Mathematics and Computer Science\\
Dickinson College\\ Carlisle, PA 17013, U.S.A.}
\email{schaefje@dickinson.edu}

\begin{abstract}
Let $k$ be an algebraically closed field of positive characteristic, and let $W$ be the ring of infinite Witt 
vectors over $k$. Suppose $G$ is a finite group and  $B$ is a block of $kG$ of infinite tame 
representation type.
We find all finitely generated $kG$-modules $V$ that belong to $B$ and whose endomorphism ring 
is isomorphic to $k$ and determine the universal deformation ring $R(G,V)$ for each of these modules.
\end{abstract}

\maketitle


\section{Introduction}
\label{s:intro}

Let $k$ be an algebraically closed field of characteristic $p>0$ and let $\mathcal{O}$ be a complete 
discrete valuation ring of characteristic $0$ with residue field $k$. Suppose $G$ is a finite group and 
$V$ is a finitely generated $kG$-module. It is a classical question to ask whether $V$ can be lifted to an
$\mathcal{O}$-free $\mathcal{O}G$-module. In \cite{green}, Green showed that this is 
always possible when $\mathrm{Ext}^2_{kG}(V,V)=0$. However, there are many cases when this Ext 
group is not zero and $V$ can still be lifted over $\mathcal{O}$. This lifting question can be seen as a 
special case of a more general deformation question which asks over which complete local commutative 
Noetherian $\mathcal{O}$-algebras with residue field $k$ the $kG$-module $V$ can be lifted. Since $k$ 
is algebraically closed, one usually takes $\mathcal{O}$ to be the ring $W=W(k)$ of infinite Witt vectors
over $k$. It was shown in \cite[Prop. 2.1]{bc}
that if the  \emph{stable} endomorphism ring of $V$ is isomorphic to $k$, then $V$ has a so-called 
universal deformation ring $R(G,V)$. This ring is universal in the sense that every isomorphism class of
lifts of $V$ over a complete local commutative Noetherian ring $R$ with residue field $k$ is associated to
a unique morphism $R(G,V)\to R$ (see Section \ref{s:prelim}).

Suppose that the stable endomorphism ring of $V$ is isomorphic to $k$.
In \cite{bc} (resp. \cite{bl}), the isomorphism types of $R(G,V)$ were determined for all such $V$ 
belonging to a cyclic block (resp. to a block with Klein four defect groups). In \cite{diloc,3sim,2sim}, 
the rings $R(G,V)$ were determined for all such $V$ belonging to various tame blocks with dihedral 
defect groups. By \cite{gowa}, these blocks include in particular all  blocks 
that are Morita equivalent to principal blocks with dihedral defect groups.
For other tame blocks, however, usually much less is known with respect to their representation
theory. For this reason, it still remains to systematically study all blocks of infinite tame representation
type with respect to universal deformation rings, and this is the goal of the present paper.
The key tools used to determine the universal deformation rings in all of the above cases are results
from modular and ordinary representation theory due to Brauer, Erdmann \cite{erd}, 
Linckelmann \cite{linckel,linckel1}, Carlson-Th\'{e}venaz \cite{carl2}, and others.

The main motivation for studying universal deformation rings for finite groups is that this case helps
understand ring theoretic properties of universal deformation rings for profinite groups $\Gamma$.
The latter have become an important tool in number theory, in particular if $\Gamma$ is a
profinite Galois group (see e.g. \cite{wita,wi}, \cite{bcdt}, \cite{khawin} and their references).
In \cite{lendesmit}, de Smit and Lenstra showed
that if $\Gamma$ is an arbitrary profinite group and $V$ is a finite
dimensional vector space over $k$ with a continuous $\Gamma$-action which has a universal
deformation ring  $R(\Gamma,V)$, then $R(\Gamma,V)$ is the inverse limit of the universal 
deformation rings $R(G,V)$ when $G$ ranges over all finite discrete quotients of $\Gamma$ through 
which the $\Gamma$-action on $V$ factors. Thus to answer questions about the ring structure of 
$R(\Gamma,V)$, it is natural to first consider the case when $\Gamma=G$ is finite.
When determining $R(G,V)$, the main advantage is that one can make use of powerful techniques 
that are not available for arbitrary profinite groups $\Gamma$,
such as decomposition matrices, Auslander-Reiten theory and the Green correspondence.

Suppose now that $B$ is a block of $kG$ of infinite tame representation type.  In
\cite{erd}, Erdmann gave a list of all possible quivers and relations 
which determine the basic algebra $\Lambda$ of $B$ up to isomorphism.
In the case when the defect groups of $B$ are dihedral, she moreover showed
that $\Lambda/\mathrm{soc}(\Lambda)$ is a special biserial algebra. This means that
in this case one can give a complete list of isomorphism classes of indecomposable 
$\Lambda$-modules using so-called strings and bands  (see \cite{buri}).
In particular, this made it possible in \cite{diloc,3sim,2sim} to determine all $B$-modules
whose \emph{stable} endomorphism rings are isomorphic to $k$ when $B$ has dihedral defect groups. 
For arbitrary blocks $B$ of infinite tame representation type, one usually cannot give
such a complete list. 
However,  we will show that it is still possible to determine all isomorphism classes of 
$B$-modules whose endomorphism rings are isomorphic to $k$. 

Our main result is as follows; more precise statements can be found in 
Lemma \ref{lem:local} and Theorem \ref{thm:maindetail}.

\begin{thm}
\label{thm:bigmain}
Suppose $G$ is a finite group, $B$ is a block of $kG$ of infinite tame representation type, and $D$ is a
defect group of $B$ of order $p^n$. Let $V$ be a $kG$-module belonging to $B$ whose endomorphism
ring is isomorphic to $k$, and let $R(G,V)$ be its universal deformation ring. 
Let $d^1(V)=\mathrm{dim}_k\,\mathrm{Ext}^1_{kG}(V,V)$. Then $d^1(V) \in\{0,1,2\}$.
\begin{enumerate}
\item[(i)] If $d^1(V)=0$, then either $R(G,V)\cong W$ or $R(G,V)\cong k$.
\item[(ii)] If $d^1(V)=1$, then either
	\begin{enumerate}
	\item[(a)] $R(G,V) \cong W[[t]]/(t^p-p\mu\, t))$ for some non-zero $\mu\in W$, or
	\item[(b)] $R(G,V) \cong W[[t]]/(t^p,p\,t)$, or
	\item[(c)] $n\ge 4$ and
	there exists a monic polynomial $q_n(t)\in W[t]$ of degree $p^{n-2}-1$,
	which depends only on $D$ and which can be given explicitly, such that either
	$$R(G,V) \cong W[[t]]/(q_n(t))\mbox{ or }R(G,V) \cong W[[t]]/(t\,q_n(t)),p\,q_n(t)).$$
	\end{enumerate}
\item[(iii)] If $d^1(V)=2$, then $R(G,V)\cong W[[t_1,t_2]]/(t_1^p-p\,t_1,t_2^p-p\,t_2)$.
\end{enumerate}
In all cases, $R(G,V)$ is isomorphic to a subquotient algebra of the group algebra $WD$,
giving a positive answer to \cite[Question 1.1]{bc}.
\end{thm}

To prove Theorem \ref{thm:bigmain}, we first determine all $B$-modules $V$ whose endomorphism
rings are isomorphic to $k$ by finding the $\Lambda$-modules $M$ that correspond to $V$ under
the Morita equivalence between $B$ and its basic algebra $\Lambda$. The main idea is
to use the decription of the projective indecomposable modules to classify
certain $\Lambda$-modules that have a short radical series. It turns out that the $\Lambda$-modules
$M$ we need to find have at most 4 composition factors, resulting in a finite list of isomorphism classes of
$B$-modules $V$ whose endomorphism rings are isomorphic to $k$.
We then determine the universal deformation ring $R(G,V)$ for each of these modules $V$.
Computing $d^1(V)$ shows that the case $d^1(V)=2$ only occurs when $B$ is local, i.e. 
when there is a unique isomorphism class of simple $B$-modules. This allows us to use 
nilpotent blocks to prove part (iii) of Theorem \ref{thm:bigmain}. For non-local $B$, 
$R(G,V)$ is determined in two steps:
Using the basic algebra $\Lambda$, we first determine the universal
mod $p$ deformation ring $R(G,V)/p R(G,V)$. Using decomposition matrices
and generalized decomposition numbers, we then determine the full universal deformation ring 
$R(G,V)$.
In particular, we use the results from \cite{brauerpaper} to prove part (ii)(c) of Theorem \ref{thm:bigmain}.

The paper is organized as follows. 
In Section  \ref{s:prelim}, we review the basic definitions and results
concerning universal deformation rings of modules for finite groups. 
In Section \ref{s:tame}, we let $B$ be a block of infinite tame representation type and
set up the notation for the remainder of the paper.
We also deal with the case when $B$ is a local block
(see Lemma \ref{lem:local}). 
For the remainder of the paper, we let $B$ be non-local.
In Section \ref{s:endo}, we determine all $B$-modules $V$ whose endomorphism rings are isomorphic
to $k$ (see Proposition \ref{prop:endok}).
In Section \ref{s:udr}, we then determine the universal deformation ring $R(G,V)$ for each
such module $V$ (see Theorem \ref{thm:maindetail}). This, together with
Lemma \ref{lem:local}, proves Theorem \ref{thm:bigmain}.


\section{Preliminaries}
\label{s:prelim}

In this section, we give a brief introduction to deformation rings and deformations. 
For more background material, we refer the reader to \cite{maz1} and \cite{lendesmit}.

Let $k$ be an algebraically closed field of characteristic $p>0$, and let $W=W(k)$ be the ring of infinite 
Witt vectors over $k$.
Let $\hat{\mathcal{C}}$ be the category of all complete local commutative Noetherian 
rings with residue field $k$. Note that all rings in $\hat{\mathcal{C}}$ have a natural $W$-algebra structure.
The morphisms in $\hat{\mathcal{C}}$ are continuous $W$-algebra 
homomorphisms which induce the identity map on $k$.

Suppose $G$ is a finite group and $V$ is a finitely generated $kG$-module. 
A \emph{lift} of $V$ over an object $R$ in $\hat{\mathcal{C}}$ is a pair $(M,\phi)$ where $M$ is a finitely 
generated $RG$-module which is free over $R$, and $\phi:k\otimes_R M\to V$ is an isomorphism of 
$kG$-modules. Two lifts $(M,\phi)$ and $(M',\phi')$ of $V$ over $R$ are isomorphic if there is an 
isomorphism $f:M\to M'$ with $\phi=\phi'\circ (k\otimes f)$. The isomorphism class $[M,\phi]$ of a lift 
$(M,\phi)$ of $V$ over $R$ is called a \emph{deformation} of $V$ over $R$, and the set of all such deformations 
is denoted by $\mathrm{Def}_G(V,R)$. The deformation functor
$$\hat{F}_V:\hat{\mathcal{C}} \to \mathrm{Sets}$$ 
is a covariant functor which
sends an object $R$ in $\hat{\mathcal{C}}$ to $\mathrm{Def}_G(V,R)$ and a morphism 
$\alpha:R\to R'$ in $\hat{\mathcal{C}}$ to the map $\mathrm{Def}_G(V,R) \to
\mathrm{Def}_G(V,R')$ defined by $[M,\phi]\mapsto [R'\otimes_{R,\alpha} M,\phi_\alpha]$, where  
$\phi_\alpha=\phi$ after identifying $k\otimes_{R'}(R'\otimes_{R,\alpha} M)$ with $k\otimes_R M$.

Suppose there exists an object $R(G,V)$ in $\hat{\mathcal{C}}$ and a deformation 
$[U(G,V),\phi_U]$ of $V$ over $R(G,V)$ with the following property:
For each $R$ in $\hat{\mathcal{C}}$ and for each lift $(M,\phi)$ of $V$ over $R$ there exists 
a morphism $\alpha:R(G,V)\to R$ in $\hat{\mathcal{C}}$ such that $\hat{F}_V(\alpha)([U(G,V),\phi_U])=
[M,\phi]$, and moreover $\alpha$ is unique if $R$ is the ring of dual numbers
$k[\epsilon]/(\epsilon^2)$. Then $R(G,V)$ is called the \emph{versal deformation ring} of $V$ and 
$[U(G,V),\phi_U]$ is called the \emph{versal deformation} of $V$. If the morphism $\alpha$ is
unique for all $R$ and all lifts $(M,\phi)$ of $V$ over $R$, 
then $R(G,V)$ is called the \emph{universal deformation ring} of $V$ and $[U(G,V),\phi_U]$ is 
called the \emph{universal deformation} of $V$. In other words, $R(G,V)$ is universal if and only if
$R(G,V)$ represents the functor $\hat{F}_V$ in the sense that $\hat{F}_V$ is naturally isomorphic to 
the Hom functor $\mathrm{Hom}_{\hat{\mathcal{C}}}(R(G,V),-)$. 

By \cite{maz1}, every finitely generated $kG$-module $V$ has a versal deformation ring $R(G,V)$.
By \cite[Prop. 2.1]{bc}, if the {\bf stable} endomorphism ring $\underline{\mathrm{End}}_{kG}(V)$ 
is isomorphic to $k$, then $R(G,V)$ is universal.

Note that the above definition of deformations can be weakened as follows.
Given a lift $(M,\phi)$ of $V$ over a ring $R$ in 
$\hat{\mathcal{C}}$, define the corresponding \emph{weak deformation} to be
the isomorphism class of $M$ as an $RG$-module, without taking into account the specific 
isomorphism $\phi:k\otimes_RM\to V$. 
In general, a weak deformation of $V$ over $R$ identifies more lifts than a deformation of $V$ 
over $R$
that respects the isomorphism $\phi$ of a representative $(M,\phi)$.
However, if the stable
endomorphism ring $\underline{\mathrm{End}}_{kG}(V)$ is isomorphic to $k$, these two 
definitions
of deformations coincide (see  \cite[Remark 2.1]{3quat}).


\section{Tame blocks}
\label{s:tame}

We make the following assumptions for the remainder of the paper:

\begin{hypo}
\label{hyp:alltheway}
Let $k$ be an algebraically closed field of positive characteristic $p$, and let 
$W=W(k)$ be the ring of infinite Witt vectors over $k$.
Suppose $G$ is a finite group, $B$ is a block of $kG$ of infinite tame representation type, 
and $D$ is a defect group of $B$ of order $p^n$. 
\end{hypo}

It follows from \cite{bd,br,hi} that $p=2$, $n\ge 2$, and $D$ is dihedral, semidihedral or generalized 
quaternion. 
In particular, we have $n\ge 2$ if $D$ is dihedral, $n\ge 3$ if $D$ is generalized quaternion, 
and $n\ge 4$ if $D$ is semidihedral.
By \cite{brIV,brauer2,olsson}, it follows that there are at most three isomorphism classes of simple
$B$-modules. 

We first consider the case when $B$ in Hypothesis \ref{hyp:alltheway} is local, i.e. when there is precisely
one isomorphism class of simple $B$-modules. 
We obtain the following result.

\begin{lemma}
\label{lem:local}
Assume Hypothesis $\ref{hyp:alltheway}$, and that $B$ is local. Let $S$ be a simple $kG$-module
belonging to $B$. Then $S$ is, up to isomorphism, the only $kG$-module belonging to $B$ whose 
endomorphism ring is isomorphic to $k$.
We have $\mathrm{dim}_k\,\mathrm{Ext}^1_{kG}(S,S)=2$ and
$R(G,S)\cong W[\mathbb{Z}/2\times\mathbb{Z}/2]$. In particular,
$R(G,S)\cong W[[t_1,t_2]]/(t_1^p-p\,t_1,t_2^p-p\,t_2)$ and $R(G,S)$ is isomorphic to 
a subquotient algebra of $WD$.
\end{lemma}

\begin{proof}
Recall that $p=2$.
Since $B$ is local and of infinite tame representation type, it follows that $B$ is \emph{nilpotent} 
in the sense of \cite{brouepuig} (see e.g. \cite[Sect. 2.5]{linckelkessar}). Let $\hat{B}$ be the block 
of $WG$ corresponding to $B$. Then $\hat{B}$ is also nilpotent. The main result of \cite{puig}
implies that $B$ is Morita equivalent to $kD$ and $\hat{B}$ is Morita equivalent to $WD$
(see \cite[Sect. 1.4]{puig}). 
Since every non-zero $B$-module has a non-zero socle and a non-zero radical quotient, it is
immediate that, up to isomorphism, the only $kG$-module belonging to $B$ whose 
endomorphism ring is isomorphic to $k$ is $S$. Using the Morita equivalence between
$\hat{B}$ and $WD$, it follows for example from \cite[Prop. 2.5]{bl} that $R(G,S)\cong
R(D,k)$ when $k$ denotes the trivial simple $kD$-module (which corresponds to $S$ under
the Morita equivalence). By \cite[Sect. 1.4]{maz1}, $R(D,k)$ is isomorphic to the group ring over $W$ of
the maximal abelian $p$-quotient of $D$. Since $p=2$ and $D$ is dihedral, semidihedral or
generalized quaternion, the maximal abelian $2$-quotient of $D$ is isomorphic to 
$\mathbb{Z}/2\times \mathbb{Z}/2$, which proves Lemma \ref{lem:local}.
\end{proof}

Assume Hypothesis \ref{hyp:alltheway}, and that $B$ is \textbf{non-local}.
From Erdmann's classification of all blocks of tame representation type 
in \cite{erd}, it follows that the quiver and relations of the basic algebra
of $B$ can be given explicitly and that, up to Morita equivalence, there
are 24 families of non-local blocks $B$. 
We use the description of these families as given in 
\cite[Sect. 4]{brauerpaper}, where Erdmann's results in \cite{erdsemid,erd} and 
\cite[Prop. 4.2]{holm} were combined with Eisele's results in \cite{eisele}, giving the
list in Figure \ref{fig:list}. Note that $\D$, or $\SD$, or $\Q$ in the name indicates that the defect groups of
$B$ are dihedral, or semidihedral, or generalized quaternion, respectively.

\begin{figure}[ht] 
\caption{\label{fig:list} The list of basic algebras from \cite[Sect. 4]{brauerpaper}
that are Morita equivalent to a non-local block $B$ satisfying Hypothesis \ref{hyp:alltheway}.}
\bigskip
\begin{itemize}
\item 
$\D(2\mathcal{A})$, $\D(2\mathcal{B})$, $\D(3\mathcal{A})_1$, $\D(3\mathcal{B})_1$, $\D(3\mathcal{K})$;
\item 
$\SD(2\mathcal{A})_1(c)$, $\SD(2\mathcal{A})_2(c)$, 
$\SD(2\mathcal{B})_1(c)$, $\SD(2\mathcal{B})_2(c)$, $\SD(2\mathcal{B})_4(c)$,
$\SD(3\mathcal{A})_1$, $\SD(3\mathcal{B})_1$, $\SD(3\mathcal{B})_2$,
$\SD(3\mathcal{C})_{2,1}$, $\SD(3\mathcal{C})_{2,2}$, $\SD(3\mathcal{D})$,
$\SD(3\mathcal{H})_1$, $\SD(3\mathcal{H})_2$;
\item 
$\Q(2\mathcal{A})(c)$, $\Q(2\mathcal{B})_1(c)$, $\Q(2\mathcal{B})_2(p,a,c)$,
$\Q(3\mathcal{A})_2$, $\Q(3\mathcal{B})$, $\Q(3\mathcal{K})$.
\end{itemize}
\end{figure}

We will also make use of the decomposition matrix for each non-local block $B$, including
the order of the ordinary irreducible characters, as given in \cite[Appendix]{brauerpaper}.
Note that $B$ always contains exactly 4 ordinary irreducible characters of height 0 and, unless
$D$ is quaternion of order 8, exactly $2^{n-2}-1$ ordinary irreducible characters of height 1. 
If $D$ is quaternion of order 8, $B$ contains exactly 3 ordinary irreducible characters of height 1. 
If $n\ge 4$ then the family of $2^{n-2}-1$ ordinary irreducible characters of height 1
all define the same Brauer character on restricting to the 2-regular conjugacy classes of $G$. 
If $D$ is generalized quaternion or semidihedral, there may be additional ordinary 
irreducible characters of height $n-2$. In the decomposition matrices in \cite[Appendix]{brauerpaper}, 
the 4 ordinary irreducible characters of height 0 are listed first, followed by the family of $2^{n-2}-1$ 
ordinary irreducible characters of height 1, and finally the ordinary irreducible characters 
of height $n-2$ if they exist.

For each algebra $\Lambda$ in Figure \ref{fig:list}, we use the following notation for certain modules
of small length.

\begin{dfn}
\label{def:small}
Assume Hypothesis $\ref{hyp:alltheway}$, and that $B$ is non-local.
Let $\Lambda=kQ/I$ be a basic algebra such that
$B$ is Morita equivalent to $\Lambda$, where we assume $\Lambda$ is one of the algebras
in Figure \ref{fig:list}.
For each vertex $j$ in $Q$, let $S_j$ denote a simple $\Lambda$-module corresponding to $j$.
\begin{enumerate}
\item[(a)] Let $v_1,v_2,\ldots,v_\ell$ be (not necessarily distinct) vertices of $Q$. If there 
exists, up to isomorphism, a unique uniserial $\Lambda$-module with descending composition
factors $S_{v_1},S_{v_2},\ldots,S_{v_\ell}$, we denote such a $\Lambda$-module by
$$S_{v_1,v_2,\ldots,v_\ell}=\begin{array}{c} S_{v_1}\\S_{v_2}\\:\\:\\S_{v_\ell}\end{array}$$

\item[(b)] Let $u,v,w$ be (not necessarily distinct) vertices of $Q$. If there 
exists, up to isomorphism, a unique indecomposable $\Lambda$-module with descending radical
factors $S_u,S_v\oplus S_w$, we denote such a $\Lambda$-module by
$$T_{u,v\oplus w}=\begin{array}{cc}\multicolumn{2}{c}{S_u}\\S_v&S_w\end{array}$$
If there 
exists, up to isomorphism, a unique indecomposable $\Lambda$-module with descending radical
factors $S_v\oplus S_w,S_u$, we denote such a $\Lambda$-module by
$$T_{v\oplus w,u}=\begin{array}{cc}S_v&S_w\\\multicolumn{2}{c}{S_u}\end{array}$$
\end{enumerate}
\end{dfn}


\section{Modules with endomorphism ring $k$}
\label{s:endo}

We assume Hypothesis \ref{hyp:alltheway}, and that $B$ is \textbf{non-local}. 
In this section, we determine all finitely generated
$B$-modules whose endomorphism ring is isomorphic to $k$.

\begin{prop}
\label{prop:endok}
Assume Hypothesis $\ref{hyp:alltheway}$, and that $B$ is non-local.
Let $\Lambda=kQ/I$ be a basic algebra such that
$B$ is Morita equivalent to $\Lambda$, where we assume $\Lambda$ is one of the algebras
in Figure $\ref{fig:list}$.
Let $\mathcal{E}$ be a complete set of representatives of non-isomorphic 
$kG$-modules $V$ belonging to $B$ with $\mathrm{End}_{kG}(V)\cong k$.
Let $\mathcal{E}_\Lambda$ be a set of $\Lambda$-modules
that correspond to the modules in $\mathcal{E}$ under the Morita equivalence between
$B$ and $\Lambda$.
Using the notation from Definition $\ref{def:small}$, $\mathcal{E}_\Lambda$ is given as follows:

\begin{enumerate}
\item[(i)] If $Q\in\{2\mathcal{A},2\mathcal{B}\}$ and $\Lambda\not\in
\{\SD(2\mathcal{B})_4(c),\Q(2\mathcal{B})_2(p,a,c)\}$, then 
$$\mathcal{E}_\Lambda=\{S_0, S_1,S_{01},S_{10},S_{001},S_{100}\}.$$
If $\Lambda\in\{\SD(2\mathcal{B})_4(c),\Q(2\mathcal{B})_2(p,a,c)\}$, then 
$$\mathcal{E}_\Lambda=\{S_0, S_1, S_{01},S_{10}\}.$$

\item[(ii)] If $Q\in\{3\mathcal{A},3\mathcal{B}\}$ and $\Lambda\neq \SD(3\mathcal{B})_1$, then
$$\mathcal{E}_\Lambda=\{S_0, S_1, S_2,S_{01},S_{10},S_{02},S_{20},
S_{102},S_{201},S_{0102},S_{2010},S_{0201},S_{1020},T_{0,1\oplus 2},T_{1\oplus 2,0}\}.$$
If $\Lambda=\SD(3\mathcal{B})_1$, then 
$$\mathcal{E}_\Lambda=\{S_0, S_1, S_2,S_{01},S_{10},S_{02},S_{20},
S_{102},S_{201},S_{0201},S_{1020},T_{0,1\oplus 2},T_{1\oplus 2,0}\}.$$

\item[(iii)] If $Q=3\mathcal{C}$, then 
$$\mathcal{E}_\Lambda=\{S_0, S_1, S_2,S_{01},S_{10},S_{02},S_{20},
S_{102},S_{201},T_{0,1\oplus 2},T_{1\oplus 2,0}\}.$$

\item[(iv)] If $Q=3\mathcal{D}$, then 
$$\mathcal{E}_\Lambda=\{S_0, S_1, S_2,S_{01},S_{10},S_{02},S_{20},
S_{102},S_{201},S_{0102},S_{2010},T_{0,1\oplus 2},T_{1\oplus 2,0}\}.$$

\item[(v)] If $Q=3\mathcal{H}$, then 
$$\mathcal{E}_\Lambda=\{S_0, S_1, S_2,S_{01},S_{10},S_{20},S_{12},S_{21},
S_{012},T_{1\oplus 2,0},T_{1,0\oplus 2},T_{0\oplus 2,1},T_{2,0\oplus 1}\}.$$

\item[(vi)] If $Q=3\mathcal{K}$, then 
$$\mathcal{E}_\Lambda=\{S_0, S_1, S_2,S_{01},S_{10},S_{02},S_{20},S_{12},S_{21},
T_{0,1\oplus 2},T_{1\oplus 2,0},T_{1,0\oplus 2},T_{0\oplus 2,1},T_{2,0\oplus 1},T_{0\oplus 1,2}\}.$$
\end{enumerate}
\end{prop}

\begin{proof}
Proposition \ref{prop:endok}
is proved using the description of the basic algebras $\Lambda$ in Figure \ref{fig:list},
as provided in \cite[Sect. 4]{brauerpaper}. 
We illustrate the main arguments of the proof by considering the cases when $\Lambda$ is equal
to either $\SD(2\mathcal{A})_1(c)$ or $\Q(3\mathcal{B})$.

\begin{enumerate}
\item[(a)]
Suppose first that $\Lambda=\SD(2\mathcal{A})_1(c)=k[2\mathcal{A}]/I_{\SD(2\mathcal{A})_1,c}$
for some $c\in k$, where
the quiver $2\mathcal{A}$ and the ideal $I_{\SD(2\mathcal{A})_1,c}$ are as in Figure
\ref{fig:SD2A}. Note that $n\ge 4$.

\begin{figure}[ht] 
\caption{\label{fig:SD2A} The quiver and relations 
for $\Lambda=\SD(2\mathcal{A})_1(c)=k[2\mathcal{A}]/I_{\SD(2\mathcal{A})_1,c}$.}
\begin{eqnarray*}
\raisebox{-2ex}{$2\mathcal{A}$}&\raisebox{-2ex}{=}&\xymatrix @R=-.2pc {
0&1\\
\ar@(ul,dl)_{\alpha} \bullet \ar@<.8ex>[r]^{\beta} &\bullet\ar@<.9ex>[l]^{\gamma}}
\\[1ex]
I_{\SD(2\mathcal{A})_1,c}&=&\!\!\!\langle \alpha^2-c(\gamma\beta\alpha)^{2^{n-2}},
\beta\gamma\beta-\beta\alpha(\gamma\beta\alpha)^{2^{n-2}-1},\\
&& \gamma\beta\gamma-\alpha\gamma(\beta\alpha\gamma)^{2^{n-2}-1},
  \alpha(\gamma\beta\alpha)^{2^{n-2}}\rangle .
\end{eqnarray*}
\end{figure}

Let $e_0$ and $e_1$ denote the images of the primitive idempotents of $k[2\mathcal{A}]$ 
corresponding to the vertices $0$ and $1$, respectively. Let
$S_0$ and $S_1$ denote representatives of the isomorphism classes of simple 
$\Lambda$-modules. The projective indecomposable $\Lambda$-modules are pictured 
in Figure \ref{fig:projSD2A}, where we use the short-hand $0,1$ to denote $S_0,S_1$, respectively.
\begin{figure}[ht] 
\caption{\label{fig:projSD2A} The projective indecomposable modules
for $\Lambda=\SD(2\mathcal{A})_1(c)$.}
$$P_0=\vcenter{\xymatrix @R=.1pc @C=.2pc {&0&\\
1&&0\ar@{.}[ldddddddd]\\0\ar@{-}[rrdddddd]&&1\\0&&0\\ :&&:\\:&&:\\0&&0\\1&&0\\
0&&1\\&0&}},
\qquad P_1=\vcenter{\xymatrix @R=.1pc @C=.2pc {&1&\\&0\ar@{-}[lddd]&\\
&&0\\&&1\\1\ar@{-}[rdddd]&&:\\&&:\\&&1\\&&0\\&0&\\&1&}}$$
\vspace{1ex}
\end{figure}

Suppose $M$ is a non-simple $\Lambda$-module such that $\mathrm{End}_\Lambda(M)\cong k$.
Then $M/\mathrm{rad}(M)$ and $\mathrm{soc}(M)$ do not have any composition factors in common.
We first prove the following auxiliary statement:
\begin{equation}
\label{eq:needthis1}
\mbox{If $x\in M$, then $(\beta\gamma)\, x= 0$.}
\end{equation}
Suppose, by contradiction, that there exists $x\in M$ 
such that $(\beta\gamma)\, x\neq 0$. Since $(\beta\gamma)\, x = (\beta\gamma)\, e_1x$, we
replace $x$ by $e_1x$ to be able to assume that $e_1 x = x$. 
Suppose first  $(\gamma\beta\gamma)\, x\neq 0$. Then it follows from the relations in 
$\Lambda$ from Figure \ref{fig:SD2A} that
$\left(\alpha\gamma(\beta\alpha\gamma)^{2^{n-2}-1}\right) x$ is also not zero in $M$. This implies that 
$\Lambda x$, which is a submodule of $M$, is isomorphic to $P_1/\mathrm{soc}(P_1)$ or to $P_1$.
Since $M$ cannot be isomorphic to $P_1$, we obtain that $\Lambda x \cong P_1/\mathrm{soc}(P_1)$.
Therefore, using the notation from Definition \ref{def:small}, $S_{10}$ is isomorphic to a submodule 
of $M$.
On the other hand, $\Lambda x$ is isomorphic to a submodule of $P_0$ and, since $M$ is not
isomorphic to $P_0$, $\Lambda x$ is isomorphic to a submodule of $\mathrm{rad}(P_0)$. This implies 
that $x\in M-\mathrm{rad}(M)$ and that $S_{10}$ is also isomorphic to a quotient module of $M$. 
But this means that $M$ has a non-zero endomorphism factoring 
through $S_{10}$, contradicting $\mathrm{End}_\Lambda(M)\cong k$. Therefore, we must have
$(\gamma\beta\gamma)\, x =  0$. This implies that $(\beta\gamma)\,x$ lies in the socle of $M$, which 
means that $S_1$ is a direct summand of $\mathrm{soc}(M)$. In particular, it follows that 
$x\in\mathrm{rad}(M)$, since otherwise $M$ has a non-zero endomorphism factoring 
through $S_1$. Since $\mathrm{Ext}^1_\Lambda(S_i, S_1)=0$ unless $i=0$, this means there exists
$w\in M$ with $e_0 w= w$ such that $\beta\, w = x$ modulo $\mathrm {rad}^2(\Lambda w)$. Using the
relations in $\Lambda$, we see that this implies  $(\beta\gamma\beta)\,w = (\beta\gamma)\,x\neq 0$.
Using again the relations in $\Lambda$, we obtain that 
$\left(\beta\alpha(\gamma\beta\alpha)^{2^{n-2}-1}\right) w$ is also not zero in $M$. Therefore 
$\Lambda w$, which is a submodule of $M$, surjects onto a quotient module of $P_0$ of the form
$$\vcenter{\xymatrix @R=.1pc @C=.2pc {&0\ar@{-}[lddd]&\\
&&0\\&&1\\1\ar@{-}[rdddd]&&:\\&&:\\&&1\\&&0\\&0&\\&1&}}\quad\cong\quad
\mathrm{rad}(P_1)$$
Since $M$ cannot be isomorphic to $P_1$, it follows that $w\in M-\mathrm{rad}(M)$. 
In particular, this implies that $S_{01}$ is a quotient module of $M$.
Note that $\Lambda w$ is isomorphic to a quotient module of $P_0/\mathrm{soc}(P_0)$. 
Considering all the possible quotient modules of $P_0/\mathrm{soc}(P_0)$ that surject onto 
$\mathrm{rad}(P_1)$, we see that $S_{01}$ is isomorphic to a submodule of each of them.
But this means that $M$ has a non-zero endomorphism factoring 
through $S_{01}$, contradicting $\mathrm{End}_\Lambda(M)\cong k$. 
This completes the proof of (\ref{eq:needthis1}).

Note that (\ref{eq:needthis1}) implies that the uniserial module $S_{101}$
is not isomorphic to either a submodule or
a quotient module of $M$.

Since $M/\mathrm{rad}(M)$ and $\mathrm{soc}(M)$ do not have any composition factors in common,
there are two cases:
either $M/\mathrm{rad}(M)\cong (S_1)^r$ and $\mathrm{soc}(M)\cong (S_0)^s$, or
$M/\mathrm{rad}(M)\cong (S_0)^s$ and $\mathrm{soc}(M)\cong (S_1)^r$, for certain 
$r,s\in\mathbb{Z}^+$. 

We consider the case when $M/\mathrm{rad}(M)\cong (S_1)^r$ and 
$\mathrm{soc}(M)\cong (S_0)^s$, the other case being similar. We claim that, using the notation
from Definition \ref{def:small}, $M$ is isomorphic either to $S_{10}$ or to $S_{100}$.

To prove this claim, we use that $\mathrm{Ext}^1_\Lambda(S_i,S_j)$ is one-dimensional unless
$(i,j)=(1,1)$, in which case it is zero. This implies that 
\begin{eqnarray}
\label{M:top}
M/\mathrm{rad}^2(M) &\cong& \left(\begin{array}{c}S_1\\S_0\end{array}\right)^{r_1} \oplus (S_1)^{r_2},\\
\label{M:soc1}
\mathrm{soc}_2(M) &\cong& \left(\begin{array}{c}S_1\\S_0\end{array}\right)^{s_1} \oplus
\left(\begin{array}{c}S_0\\S_0\end{array}\right)^{s_2} \oplus 
\left(\begin{array}{cc}S_1&S_0\\\multicolumn{2}{c}{S_0}\end{array}\right)^{s_3} \oplus (S_0)^{s_4}
\end{eqnarray}
for certain non-negative $r_i, s_j$, where $r_1$ and at least one of $s_1,s_2,s_3$ must be positive.
Considering (\ref{M:top}) and (\ref{M:soc1}), we see that
the $k$-dimension of $\mathrm{End}_\Lambda(M)$ is at least 2 unless
$M\cong S_{10}$ or 
\begin{equation}
\label{M:soc2}
\mathrm{soc}_2(M) \cong \left(\begin{array}{c}S_0\\S_0\end{array}\right)^{s_2} \oplus (S_0)^{s_4}
\end{equation}
where $s_2>0$.
Hence we only need to consider the case when $M$ satisfies both $(\ref{M:top})$ and
$(\ref{M:soc2})$. Since $\mathrm{Ext}^1(S_i,\begin{array}{c}S_0\\S_0\end{array})$ is one-dimensional
when $i=1$ and zero when $i=0$, it follows that
\begin{equation}
\label{M:soc3}
\mathrm{soc}_3(M) \cong \left(\begin{array}{c}S_1\\S_0\\S_0\end{array}\right)^{s_5} \oplus
\left(\begin{array}{c}S_0\\S_0\end{array}\right)^{s_6} \oplus (S_0)^{s_7}
\end{equation}
where $s_5>0$.
Since $\mathrm{Ext}^1(\begin{array}{c}S_1\\S_0\end{array},S_j)$ is one-dimensional
for both $j=0$ and $j=1$, it follows that the possible direct summands of $M/\mathrm{rad}^3(M)$
are isomorphic to
$$\begin{array}{c}S_1\\S_0\\S_0\end{array},\quad
\begin{array}{c@{}cc}S_1\\&S_0&S_1\\&\multicolumn{2}{c}{S_0}\end{array}, \quad
\begin{array}{c}S_1\\S_0\\S_1\end{array},\quad
\begin{array}{cc}\multicolumn{2}{c}{S_1}\\\multicolumn{2}{c}{S_0}\\S_1&S_0\end{array},\quad
\begin{array}{c@{}cc}&S_1\\&S_0&S_1\\S_1&\multicolumn{2}{c}{S_0}\end{array},\quad
\begin{array}{c}S_1\\S_0\end{array},\quad S_1$$
where at least one summand of radical length 3 occurs. 
Since by (\ref{eq:needthis1}) $M$ does not surject onto $S_{101}$, we obtain
$$M/\mathrm{rad}^3(M) \cong \left(\begin{array}{c}S_1\\S_0\\S_0\end{array}\right)^{r_3} \oplus
\left(\begin{array}{c@{}cc}S_1\\&S_0&S_1\\&\multicolumn{2}{c}{S_0}\end{array}\right)^{r_4} \oplus
\left(\begin{array}{c}S_1\\S_0\end{array}\right)^{r_5} \oplus (S_1)^{r_6}$$
where either $r_3>0$ or $r_4>0$. If $r_3>0$, then either $M\cong S_{100}$
or the endomorphism ring of $M$ has $k$-dimension at
least 2. Hence we only need to consider the case when 
\begin{equation}
\label{M:rad5}
M/\mathrm{rad}^3(M) \cong 
\left(\begin{array}{c@{}cc}S_1\\&S_0&S_1\\&\multicolumn{2}{c}{S_0}\end{array}\right)^{r_4} \oplus
\left(\begin{array}{c}S_1\\S_0\end{array}\right)^{r_5} \oplus (S_1)^{r_6}
\end{equation}
and $r_4>0$. Using additional $\mathrm{Ext}^1$ arguments, we see that then
$\begin{array}{c@{}cc}S_1\\&S_0&S_1\\&\multicolumn{2}{c}{S_0}\\&\multicolumn{2}{c}{S_1}\end{array}$
has to be a direct summand of $M/\mathrm{rad}^4(M)$. But this implies that there 
exists an element $x\in M$ with $(\beta\gamma)\, x\neq 0$, which contradicts (\ref{eq:needthis1}). 
Summarizing, if $M/\mathrm{rad}(M)\cong (S_1)^r$ and $\mathrm{soc}(M)\cong (S_0)^s$, then 
$M$ is isomorphic either to $S_{10}$ or to $S_{100}$.
This completes the case when $B$ is Morita equivalent to $\SD(2\mathcal{A})_1(c)$.

\medskip
 
\item[(b)]
Suppose next that $B$ is Morita equivalent to 
$\Lambda=\Q(3\mathcal{B})=k[3\mathcal{B}]/I_{\Q(3\mathcal{B})}$ where
the quiver $3\mathcal{B}$ and the ideal $I_{\Q(3\mathcal{B})}$ are as in Figure
\ref{fig:Q3B}. Note that $n\ge 4$.

\begin{figure}[ht] 
\caption{\label{fig:Q3B} The quiver and relations 
for $\Lambda=\Q(3\mathcal{B})=k[3\mathcal{B}]/I_{\Q(3\mathcal{B})}$.}
\begin{eqnarray*}
\raisebox{-2ex}{$3\mathcal{B}$}&\raisebox{-2ex}{=}&
\xymatrix @R=-.2pc {
1&0&\\
 \ar@(ul,dl)_{\alpha} \bullet \ar@<.8ex>[r]^{\beta} \ar@<.9ex>[r];[]^{\gamma}
& \bullet \ar@<.8ex>[r]^(.46){\delta} \ar@<.9ex>[r];[]^(.54){\eta} & \bullet\;2}
\\[1ex]
I_{\Q(3\mathcal{B})}&=&\!\!\!\langle \gamma\beta-\alpha^{2^{n-2}-1},
\alpha\gamma-\gamma\eta\delta(\beta\gamma\eta\delta), 
\beta\alpha-\eta\delta\beta(\gamma\eta\delta\beta),\\
&&\delta\eta\delta-\delta\beta\gamma(\eta\delta\beta\gamma),
\eta\delta\eta-\beta\gamma\eta(\delta\beta\gamma\eta), \beta\alpha^2,\delta\eta\delta\beta\rangle.
\end{eqnarray*}
\end{figure}

Let $e_0$, $e_1$ and $e_2$ denote the images of the primitive idempotents of $k[3\mathcal{B}]$ 
corresponding to the vertices $0$, $1$ and $2$, respectively. 
Let $S_0$, $S_1$ and $S_2$ denote representatives of the isomorphism classes of simple 
$\Lambda$-modules. The projective indecomposable $\Lambda$-modules are pictured 
in Figure \ref{fig:projQ3B}, where we use the short-hand $0,1,2$ to denote $S_0,S_1,S_2$, respectively.
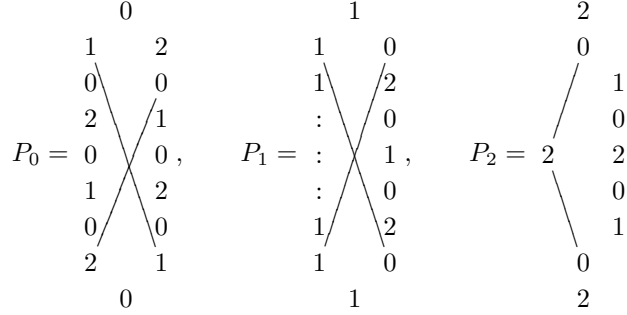
\begin{figure}[ht] 
\caption{\label{fig:projQ3B} The projective indecomposable modules
for $\Lambda=\Q(3\mathcal{B})$.}
$$P_0=\vcenter{\xymatrix @R=.1pc @C=.2pc {&0&\\
1\ar@{-}[rrdddddd]&&2\\0&&0\ar@{-}[llddddd]\\2&&1\\ 0&&0\\1&&2\\0&&0\\
2&&1\\&0&}},
\qquad P_1=\vcenter{\xymatrix @R=.1pc @C=.2pc {&1&\\1\ar@{-}[rrdddddd]&&0\ar@{-}[lldddddd]\\ 
1&&2\\:&&0\\ :&&1\\ :&&0\\ 1&&2\\1&&0\\&1&}},
\qquad P_2=\vcenter{\xymatrix @R=.1pc @C=.2pc {&2&\\&0\ar@{-}[lddd]&\\
&&1\\&&0\\2\ar@{-}[rddd]&&2\\&&0\\&&1\\&0&\\&2&}}$$
\vspace{1ex}
\end{figure}

Suppose $M$ is a non-simple $\Lambda$-module such that $\mathrm{End}_\Lambda(M)\cong k$.
Then $M/\mathrm{rad}(M)$ and $\mathrm{soc}(M)$ do not have any composition factors in common.
We first prove the following two auxiliary statements:
\begin{equation}
\label{eq:needthis2B}
\mbox{If $x\in M$, then $(\delta\eta)\, x= 0$.}
\end{equation}
\begin{equation}
\label{eq:needthis2A}
\mbox{If $y\in M-\mathrm{rad}(M)$, then $\alpha\, y= 0$.}
\end{equation}
The statement (\ref{eq:needthis2B}) is proved similarly to the statement (\ref{eq:needthis1}).
To prove (\ref{eq:needthis2A}), suppose, by contradiction, that there exists $y\in M-\mathrm{rad}(M)$ 
such that $\alpha\, y\neq 0$. Since $\alpha\, y = \alpha\, e_1y$, we
replace $y$ by $e_1y$ to be able to assume that $e_1 y = y$. This implies in particular that
$S_1$ is a direct summand of $M/\mathrm{rad}(M)$.
If $\alpha^2\,y\neq 0$, then there exists an integer $a\ge 2$ with
$\alpha^a\, y\neq 0$ and $\alpha^{a+1}\,y=0$. This means that $\alpha^a\, y$ lies in the socle of
$M$, implying that $S_1$ is a direct summand of $\mathrm{soc}(M)$, contradicting
$\mathrm{End}_\Lambda(M)\cong k$. Hence
$\alpha^2\,y=0$. If $(\beta\alpha)\,y=0$, then $\alpha\, y$ lies in the socle of $M$, again implying
that $S_1$ is a direct summand of $\mathrm{soc}(M)$. Therefore, $\alpha^2\,y=0$ and $(\beta\alpha)\,y
\neq 0$. But then it follows from the relations in $\Lambda$ from Figure \ref{fig:Q3B} that
$\left(\eta\delta\beta(\gamma\eta\delta\beta)\right) y$ is also not zero in $M$. Since $\alpha^2\,y=0$,
this implies that $\Lambda y$, which is a submodule of $M$, is isomorphic to a quotient module
of $P_1$ of the form
$$
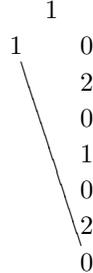
\xymatrix @R=.1pc @C=.2pc {&1&\\1\ar@{-}[rrdddddd]&&0\\ 
&&2\\&&0\\ &&1\\ &&0\\ &&2\\&&0}$$
Therefore, using the notation from Definition \ref{def:small}, $S_{10}$ 
is isomorphic to a submodule of $M$.
On the other hand, $\mathrm{Ext}^1_\Lambda(S_i, \Lambda y)=0$ unless $i=0$, 
which implies that $S_{10}$ is also a 
quotient module of $M$. But this means that $M$ has a non-zero endomorphism factoring 
through $S_{10}$, contradicting $\mathrm{End}_\Lambda(M)\cong k$. 
This proves (\ref{eq:needthis2A}).

Note that (\ref{eq:needthis2A}) implies that $S_{11}$
is not isomorphic to a submodule of $M$ and (\ref{eq:needthis2B}) implies that
$S_{202}$ is not isomorphic to either a submodule or
a quotient module of $M$.

Depending on which of $S_0,S_1,S_2$ are direct summands of $M/\mathrm{rad}(M)$ and 
$\mathrm{soc}(M)$, we obtain different possibilities for $M$. There are altogether
twelve different possibilities for $M/\mathrm{rad}(M)$ and $\mathrm{soc}(M)$.
To illustrate our arguments, we now consider two of these cases:
\begin{equation}
\label{eq:caseb1}
\mbox{$M/\mathrm{rad}(M)\cong (S_1)^r$ and $\mathrm{soc}(M)\cong (S_0)^s$ 
for certain  $r,s\in\mathbb{Z}^+$, and}
\end{equation}
\begin{equation}
\label{eq:caseb5}
\mbox{$M/\mathrm{rad}(M)\cong (S_1)^r\oplus (S_2)^s$ and
$\mathrm{soc}(M)\cong (S_0)^t$ for certain  $r,s,t\in\mathbb{Z}^+$.}
\end{equation}

Suppose first that $M$ satisfies (\ref{eq:caseb1}).
By (\ref{eq:needthis2A}), $M$ does not surject onto $S_{11}$, which implies that 
\begin{eqnarray}
\label{1M:top}
M/\mathrm{rad}^2(M) &\cong& \left(\begin{array}{c}S_1\\S_0\end{array}\right)^{r_1} \oplus (S_1)^{r_2},\\
\label{1M:soc1}
\mathrm{soc}_2(M) &\cong& \left(\begin{array}{c}S_1\\S_0\end{array}\right)^{s_1} \oplus
\left(\begin{array}{c}S_2\\S_0\end{array}\right)^{s_2} \oplus 
\left(\begin{array}{cc}S_1&S_2\\\multicolumn{2}{c}{S_0}\end{array}\right)^{s_3} \oplus (S_0)^{s_4}
\end{eqnarray}
for certain non-negative $r_i, s_j$, where $r_1$ and at least one of $s_1,s_2,s_3$ must be positive.
This implies that the $k$-dimension of $\mathrm{End}_\Lambda(M)$ is at least 2 unless
$M\cong S_{10}$ or 
\begin{equation}
\label{1M:soc2}
\mathrm{soc}_2(M) \cong \left(\begin{array}{c}S_2\\S_0\end{array}\right)^{s_2} \oplus (S_0)^{s_4}
\end{equation}
where $s_2>0$.
Hence we can concentrate on the case when $M$ satisfies both $(\ref{1M:top})$ and
$(\ref{1M:soc2})$. In this case, we have
\begin{eqnarray}
\label{1M:rad3}
M/\mathrm{rad}^3(M) &\cong& \left(\begin{array}{c}S_1\\S_0\\S_2\end{array}\right)^{r_3} \oplus
\left(\begin{array}{c}S_1\\S_0\end{array}\right)^{r_4} \oplus (S_1)^{r_5},\\
\label{1M:soc3}
\mathrm{soc}_3(M) &\cong& \left(\begin{array}{c}S_0\\S_2\\S_0\end{array}\right)^{s_5} \oplus
\left(\begin{array}{c}S_2\\S_0\end{array}\right)^{s_6} \oplus (S_0)^{s_7}
\end{eqnarray}
where $r_3,s_5>0$. By (\ref{eq:needthis2A}), this then implies that 
\begin{eqnarray}
\label{1M:rad4}
M/\mathrm{rad}^4(M) &\cong& \left(\begin{array}{c}S_1\\S_0\\S_2\\S_0\end{array}\right)^{r_6} \oplus
\left(\begin{array}{c@{}c@{}cc}S_1\\&S_0\\&&S_2&S_1\\&&\multicolumn{2}{c}{S_0}\end{array}
\right)^{r_7}\oplus \mbox{(modules of length $\le 3$)},\\
\label{1M:soc4}
\mathrm{soc}_4(M) &\cong& \left(\begin{array}{c}S_1\\S_0\\S_2\\S_0\end{array}\right)^{s_8} \oplus
\mbox{(modules of length $\le 3$)}
\end{eqnarray}
where $s_8$ and at least one of $r_6, r_7$ is positive. 
Using additional $\mathrm{Ext}^1$ arguments, we see
that $M$ cannot have a quotient module  that has radical length 5 and that surjects onto
$X=\begin{array}{c@{}c@{}cc}S_1\\&S_0\\&&S_2&S_1\\&&\multicolumn{2}{c}{S_0}\end{array}$. 
Since the endomorphism ring of $X$ has $k$-dimension 2, this implies that $r_6$ must
be positive. Therefore, it follows that $M\cong S_{1020}$, since
$M$ always has a non-zero endomorphism factoring through this module. Summarizing, if
$M$ satisfies (\ref{eq:caseb1}), then
$M$ is isomorphic either to $S_{10}$ or to $S_{1020}$.

Next suppose that $M$ satisfies (\ref{eq:caseb5}).
Since $M$ does not surject onto $S_{11}$ by (\ref{eq:needthis2A}), we obtain
\begin{eqnarray}
\label{5M:top}
M/\mathrm{rad}^2(M) &\cong& \left(\begin{array}{c}S_1\\S_0\end{array}\right)^{r_1} \oplus (S_1)^{r_2}\
\oplus \left(\begin{array}{c}S_2\\S_0\end{array}\right)^{s_1} \oplus (S_2)^{s_2}
\oplus \left(\begin{array}{cc}S_1&S_2\\\multicolumn{2}{c}{S_0}\end{array}\right)^u,\\
\label{5M:soc1}
\mathrm{soc}_2(M) &\cong& \left(\begin{array}{c}S_1\\S_0\end{array}\right)^{t_1} \oplus
\left(\begin{array}{c}S_2\\S_0\end{array}\right)^{t_2} \oplus 
\left(\begin{array}{cc}S_1&S_2\\\multicolumn{2}{c}{S_0}\end{array}\right)^{t_3} \oplus (S_0)^{t_4}
\end{eqnarray}
for certain non-negative $r_i, s_j,t_k,u$, where at least one of $r_1,r_2,u$
and at least one of $s_1,s_2,u$  and at least one of $r_1,s_1,u$ and at least one of
$t_1,t_2,t_3$ must be positive. If $t_3$ is positive, then either 
$M\cong T_{1\oplus 2,0}\cong \begin{array}{cc}S_1&S_2\\\multicolumn{2}{c}{S_0}\end{array}$ or the
$k$-dimension of $\mathrm{End}_\Lambda(M)$ is at least 2.
Hence we can concentrate on the case when $t_3=0$. In particular, $M$ has radical length at least 3.
By (\ref{eq:needthis2A}) and (\ref{eq:needthis2B}), it follows that
\begin{equation}
\label{5M:rad3}
M/\mathrm{rad}^3(M)\cong
\left(\begin{array}{c}S_1\\S_0\\S_2\end{array}\right)^{r_3} \oplus
\left(\begin{array}{c}S_2\\S_0\\S_1\end{array}\right)^{s_3} \oplus 
\mbox{(modules of radical length $\le 2$)}
\end{equation}
where at least one of $r_3,s_3$ is positive.
Therefore, we see that the $k$-dimension of $\mathrm{End}_\Lambda(M)$ is at least 2 unless
either $s_3=0=t_1=t_3$ or $r_3=0=t_2=t_3$. 
In the first of these two cases we can argue similarly as in the case when $M$ satisfies (\ref{eq:caseb1})
to see that $M$ has a non-zero endomorphism factoring through 
$S_{1020}$.
In the second case, additional $\mathrm{Ext}^1$ arguments show that $M$ has a non-zero
endomorphism factoring through $S_{2010}$.
Summarizing, if $M$ satisfies (\ref{eq:caseb5}), then
$M\cong T_{1\oplus 2,0}$.

This concludes the proof of the two cases when $M$ satisfies either (\ref{eq:caseb1}) or
(\ref{eq:caseb5}). Hence this completes the case when $B$ is Morita equivalent to $\Q(3\mathcal{B})$.
\end{enumerate}
\end{proof}


\section{Universal deformation rings}
\label{s:udr}

We assume Hypothesis \ref{hyp:alltheway}, and that $B$ is \textbf{non-local}. 
In this section, we determine 
the universal deformation ring of every $kG$-module $V$ belonging to $B$ whose 
endomorphism ring is isomorphic to $k$. In particular, this together with Lemma \ref{lem:local}
proves Theorem \ref{thm:bigmain}.
We use the lists $\mathcal{E}$ and $\mathcal{E}_\Lambda$ obtained in Proposition \ref{prop:endok}. 

We need to subdivide these lists according to different criteria. One criterion is whether
$\mathrm{Ext}^1$ is zero or not for the modules in these lists. By \cite[Sect. 6]{brauerpaper},
it is also important to separate out the modules $V\in\mathcal{E}$ whose $2$-modular character 
is equal to the restriction to the $2$-regular conjugacy classes of an ordinary irreducible character 
of $G$ of height $1$ belonging to $B$. Moreover, \cite[Prop. 6.5]{brauerpaper} also shows that 
modules that lie at the end of 3-tubes of the stable Auslander-Reiten quiver of $B$ play a special role 
when determining their universal deformation rings.

\begin{dfn}
\label{def:organizelist}
Assume Hypothesis $\ref{hyp:alltheway}$, and that $B$ is non-local.
Let $\Lambda=kQ/I$ be a basic algebra such that
$B$ is Morita equivalent to $\Lambda$, where we assume $\Lambda$ is one of the algebras
in Figure $\ref{fig:list}$. Let $\mathcal{E}$ and $\mathcal{E}_\Lambda$ be as in Proposition 
\ref{prop:endok}. 
Define the following $4$ sublists of $\mathcal{E}$:
\begin{enumerate}
\item the sublist $\mathcal{E}_1$ of $\mathcal{E}$ consisting of those modules $V$
	such that $\mathrm{Ext}^1_{kG}(V,V)\neq 0$ and the $2$-modular character of $V$
	is equal to the restriction to the $2$-regular conjugacy classes of an ordinary 
	irreducible character of $G$ of height $1$;
\item the sublist $\mathcal{E}_2$ of $\mathcal{E}$ consisting of those modules $V$
	such that $\mathrm{Ext}^1_{kG}(V,V)\neq 0$ and $V$ does not belong to $\mathcal{E}_1$;
\item the sublist $\mathcal{E}_3$ of $\mathcal{E}$ consisting of those modules $V$
	such that $\mathrm{Ext}^1_{kG}(V,V)=0$ and $V$ belongs to a $3$-tube of the stable
	Auslander-Reiten quiver of $B$;
\item the sublist $\mathcal{E}_4$ of $\mathcal{E}$ consisting of those modules $V$
	such that $\mathrm{Ext}^1_{kG}(V,V)=0$ and $V$ does not belong to $\mathcal{E}_3$.
\end{enumerate}
For $i\in\{1,2,3,4\}$, let $\mathcal{E}_{\Lambda,i}$ be the set of $\Lambda$-modules in 
$\mathcal{E}_\Lambda$ that correspond to the modules in $\mathcal{E}_i$ 
under the Morita equivalence between $B$ and $\Lambda$.
\end{dfn}

The following lemma describes the modules in each of these sublists.

\begin{lemma}
\label{lem:organize}
Assume Hypothesis $\ref{hyp:alltheway}$, and that $B$ is non-local. Let
$\Lambda=kQ/I$ be a basic algebra such that
$B$ is Morita equivalent to $\Lambda$, where we assume $\Lambda$ is one of the algebras
in Figure $\ref{fig:list}$.
Let $\mathcal{E}_\Lambda$ and $\mathcal{E}_{\Lambda,1}, \mathcal{E}_{\Lambda,2}, 
\mathcal{E}_{\Lambda,3}, \mathcal{E}_{\Lambda,4}$ be as in Definition 
$\ref{def:organizelist}$.
\begin{enumerate}
\item[(i)] If $Q=2\mathcal{A}$, then $\mathcal{E}_{\Lambda,1}=\{S_{001},S_{100}\}$
	and $\mathcal{E}_{\Lambda,2}=\{S_0,S_{01},S_{10}\}$.
	If $\Lambda\in\{\D(2\mathcal{A}),\SD(2\mathcal{A})_2(c)\}$ then
	$\mathcal{E}_{\Lambda,3}=\{S_1\}$, and
	if $\Lambda\in\{\SD(2\mathcal{A})_1(c),\Q(2\mathcal{A})(c)\}$ then
	$\mathcal{E}_{\Lambda,3}=\emptyset$.
	
	If $Q=2\mathcal{B}$ and $\Lambda\not\in
	\{\SD(2\mathcal{B})_4(c),\Q(2\mathcal{B})_2(p,a,c)\}$, then 
	$\mathcal{E}_{\Lambda,1}=\{S_1\}$ and $\mathcal{E}_{\Lambda,2}=\{S_0,S_{01},S_{10}\}$.
	If $\Lambda\in\{\D(2\mathcal{B}),\SD(2\mathcal{B})_1(c)\}$ then
	$\mathcal{E}_{\Lambda,3}=\{S_{001},S_{100}\}$, and
	if $\Lambda\in\{\SD(2\mathcal{B})_2(c),\Q(2\mathcal{B})_1(c)\}$ then
	$\mathcal{E}_{\Lambda,3}=\emptyset$.
	If $\Lambda\in\{\SD(2\mathcal{B})_4(c),\Q(2\mathcal{B})_2(p,a,c)\}$, then
	$\mathcal{E}_{\Lambda,1}=\{S_{01},S_{10}\}$, $\mathcal{E}_{\Lambda,2}=\{S_0,S_1\}$ and 
	$\mathcal{E}_{\Lambda,3}=\emptyset$.  

\item[(ii)] If $Q=3\mathcal{A}$, then $\mathcal{E}_{\Lambda,1}=\emptyset=\mathcal{E}_{\Lambda,2}$
	in the cases when $n=2$ or when $n=3$ and $D$ is quaternion, and
	$\mathcal{E}_{\Lambda,1}=\{S_{0102},S_{2010},S_{0201},S_{1020}\}$ and 
	$\mathcal{E}_{\Lambda,2}=\emptyset$ in all other cases.
	If $\Lambda=\D(3\mathcal{A})_1$, then 
	$\mathcal{E}_{\Lambda,3}=\{S_1,S_2,S_{0102},S_{2010},S_{0201},S_{1020}\}$ in the case
	when $n=2$, and $\mathcal{E}_{\Lambda,3}=\{S_1,S_2\}$ in the case when $n\ge 3$.
	If $\Lambda=\SD(3\mathcal{A})_1$ then $\mathcal{E}_{\Lambda,3}=\{S_1\}$, and
	if $\Lambda=\Q(3\mathcal{A})_2$ then $\mathcal{E}_{\Lambda,3}=\emptyset$. 
	
	If $Q=3\mathcal{B}$, then $\mathcal{E}_{\Lambda,1}=\{S_1\}$. If $Q=3\mathcal{B}$
	and $\Lambda\neq \SD(3\mathcal{B})_2$, then $\mathcal{E}_{\Lambda,2}=\emptyset$.
	If $\Lambda=\D(3\mathcal{B})_1$ then 
	$\mathcal{E}_{\Lambda,3}=\{S_2,S_{0102},S_{2010},S_{0201},S_{1020}\}$, and
	if $\Lambda=\SD(3\mathcal{B})_1$ then $\mathcal{E}_{\Lambda,3}=\{S_{0201},S_{1020}\}$, and
	if $\Lambda=\Q(3\mathcal{B})$ then $\mathcal{E}_{\Lambda,3}=\emptyset$. 
	If $\Lambda=\SD(3\mathcal{B})_2$, then $\mathcal{E}_{\Lambda,2}=\{S_{0102},S_{2010}\}$
	and $\mathcal{E}_{\Lambda,3}=\{S_2\}$.

\item[(iii)] If $Q=3\mathcal{C}$, then $\mathcal{E}_{\Lambda,3}=\emptyset$.
	If $\Lambda=\SD(3\mathcal{C})_{2,1}$, then $\mathcal{E}_{\Lambda,1}=\{S_0\}$
	and $\mathcal{E}_{\Lambda,2}=\{T_{0,1\oplus 2},T_{1\oplus 2,0}\}$.
	If $\Lambda=\SD(3\mathcal{C})_{2,2}$, then 
	$\mathcal{E}_{\Lambda,1}=\{S_{102},S_{201},T_{0,1\oplus 2},T_{1\oplus 2,0}\}$ and
	$\mathcal{E}_{\Lambda,2}=\{S_0\}$. 

\item[(iv)] If $Q=3\mathcal{D}$, then $\mathcal{E}_{\Lambda,1}=\{S_1\}$,
	$\mathcal{E}_{\Lambda,2}=\{S_2\}$ and $\mathcal{E}_{\Lambda,3}=\{S_{0102},S_{2010}\}$.

\item[(v)] If $Q=3\mathcal{H}$, then $\mathcal{E}_{\Lambda,3}=\{S_{20}\}$.
	If $\Lambda=\SD(3\mathcal{H})_1$, then $\mathcal{E}_{\Lambda,1}=\{S_{12},S_{21}\}$ and
	$\mathcal{E}_{\Lambda,2}=\{S_{01}\}$.
	If $\Lambda=\SD(3\mathcal{H})_2$, then $\mathcal{E}_{\Lambda,1}=\{S_{01},S_{10}\}$ and
	$\mathcal{E}_{\Lambda,2}=\{S_{12}\}$.
	
\item[(vi)] If $Q=3\mathcal{K}$, then $\mathcal{E}_{\Lambda,1}=\emptyset=\mathcal{E}_{\Lambda,2}$
	in the cases when $n=2$ or when $n=3$ and $D$ is quaternion, and
	$\mathcal{E}_{\Lambda,1}=\{S_{12},S_{21}\}$ and
	$\mathcal{E}_{\Lambda,2}=\emptyset$ in all other cases.
	If $\Lambda=\D(3\mathcal{K})$, then 
	$\mathcal{E}_{\Lambda,3}=\{S_{01},S_{10},S_{12},S_{21},S_{02},S_{20}\}$
	in the case when $n=2$, and
	$\mathcal{E}_{\Lambda,3}=\{S_{01},S_{10},S_{02},S_{20}\}$ in the case when $n\ge 3$.
	If $\Lambda=\Q(3\mathcal{K})$ then $\mathcal{E}_{\Lambda,3}=\emptyset$. 
\end{enumerate}
In all cases, $\mathcal{E}_{\Lambda,4}=\mathcal{E}_\Lambda -\left(\mathcal{E}_{\Lambda,1}
\cup \mathcal{E}_{\Lambda,2}\cup \mathcal{E}_{\Lambda,3}\right)$. Moreover,
$\mathrm{dim}_k\,\mathrm{Ext}^1_\Lambda(M,M)\in\{0,1\}$ for all $M\in\mathcal{E}_\Lambda$.
\end{lemma}

\begin{proof}
Lemma \ref{lem:organize} 
is proved using the description of each basic algebra $\Lambda$ in Figure \ref{fig:list},
as provided in \cite[Sect. 4]{brauerpaper}. 
Using this description, we can readily determine the $k$-dimension of $\mathrm{Ext}^1_\Lambda(M,M)$
for all modules $M\in\mathcal{E}_\Lambda$. In particular, we see that 
$\mathrm{dim}_k\,\mathrm{Ext}^1_\Lambda(M,M)\in\{0,1\}$ for all such $M$.
The modules in $\mathcal{E}_{\Lambda,1}$ have
already been determined in \cite[Lem. 6.1]{brauerpaper}. 
Note that the cases when $n=2$, respectively $n=3$ and $D$ is quaternion, play a special role, since
in these cases $\mathcal{E}_{\Lambda,1}=\emptyset$.
The  modules $M\in\mathcal{E}_\Lambda-\mathcal{E}_{\Lambda,1}$ with 
$\mathrm{Ext}^1_\Lambda(M,M)\neq 0$
then provide $\mathcal{E}_{\Lambda,2}$. If $\Lambda$ is of quaternion type,  the stable 
Auslander-Reiten quiver $\Gamma_s(\Lambda)$ of $\Lambda$ does not contain any 3-tubes, 
which implies that $\mathcal{E}_{\Lambda,3}=\emptyset$. If $\Lambda$ is of dihedral type, then
$\Gamma_s(\Lambda)$ always contains at least one 3-tube and the modules in 
$\mathcal{E}_{\Lambda,3}$ have been determined, for example, in \cite[Sect. 4]{3sim} and 
\cite[Sect. 5]{2sim}.  
If $\Lambda$ is of semidihedral type, we consider the $\Omega^2$ orbit of the $\Lambda$-modules
$M$ in $\mathcal{E}_\Lambda$ with $\mathrm{Ext}^1_\Lambda(M,M)=0$ to determine 
$\mathcal{E}_{\Lambda,3}$. It is obvious that 
$\mathcal{E}_{\Lambda,4}=\mathcal{E}_\Lambda -\left(\mathcal{E}_{\Lambda,1}
\cup \mathcal{E}_{\Lambda,2}\cup \mathcal{E}_{\Lambda,3}\right)$, which
completes the proof of Lemma \ref{lem:organize}.
\end{proof}

Using the sublists of $\mathcal{E}$ from Definition \ref{def:organizelist}, we can now
determine the universal deformation ring for every module $V$ in $\mathcal{E}$.
For the modules $V\in\mathcal{E}_1$, the universal deformation ring depends on whether or
not $V$ corresponds to a 3-tube, as defined in \cite[Def. 6.3]{brauerpaper}. 
These $V$ were explicitly determined in \cite[Lem. 6.4]{brauerpaper}.

\begin{thm}
\label{thm:maindetail}
Assume Hypothesis $\ref{hyp:alltheway}$, and that $B$ is non-local. Let
$\mathcal{E}$ and $\mathcal{E}_1,\mathcal{E}_2,\mathcal{E}_3,\mathcal{E}_4$ be as in Definition 
$\ref{def:organizelist}$.
\begin{enumerate}
\item[(a)] Suppose $V\in \mathcal{E}_1$, and let $q_n(t)\in W[t]$ be the monic polynomial 
of degree $2^{n-2}-1$ from
\cite[Def. 5.3]{brauerpaper}. If $V$ corresponds to a $3$-tube, as defined
in \cite[Def. 6.3]{brauerpaper}, then $R(G,V)\cong W[[t]]/(t\,q_n(t),2\, q_n(t))$. Otherwise
$R(G,V)\cong W[[t]]/(q_n(t))$.
\item[(b)] Suppose $V\in\mathcal{E}_2$. If $Q\in\{2\mathcal{A},2\mathcal{B}\}$ then 
$R(G,V)\cong W[[t]]/(t^2-2\mu \,t)$ for some non-zero $\mu\in W$. 
Otherwise, $R(G,V)\cong W[[t]]/(t^2,2t)$.
\item[(c)] If $V\in \mathcal{E}_3$ then $R(G,V)\cong k$.
\item[(d)] If $V\in\mathcal{E}_4$ then $R(G,V)\cong W$.
\end{enumerate}
In all cases, the ring $R(G,V)$ is isomorphic to a subquotient ring of $WD$.
\end{thm}

\begin{proof}
Recall that $p=2$.

Part (a) of Theorem \ref{thm:maindetail} follows from \cite[Thm. 6.6]{brauerpaper}. 
Part (c) follows in the case when $D$ is dihedral from \cite[Sect. 5.2]{3sim} and \cite[Prop. 6.3]{2sim},
and in the case when $D$ is semidihedral by using similar arguments as in the proof of 
\cite[Prop. 6.5]{brauerpaper}. 

To prove parts (b) and (d), let $\Lambda=kQ/I$ be a basic algebra such that
$B$ is Morita equivalent to $\Lambda$, where we assume $\Lambda$ is one of the algebras
in Figure $\ref{fig:list}$.

To prove part (b), suppose that $V\in\mathcal{E}_2$. Since $\mathrm{Ext}^1_{kG}(V,V)\cong k$, 
$R(G,V)$ is isomorphic to a quotient algebra of $W[[t]]$. Note that each $V$ has either a
simple radical quotient or a simple socle. Considering the submodules and quotient modules
of the projective indecomposable $B$-modules,
we see that there is a unique $B$-module $\overline{U}$, up to isomorphism, such that 
we have a short exact sequence
$$0\to V \xrightarrow{\iota} \overline{U} \xrightarrow{\pi} V \to 0.$$
Therefore, $\overline{U}$ defines a lift of $V$ over $k[t]/(t^2)$ where we let $t$ act as the composition
$\iota\circ\pi$. Moreover, we see that $\iota(V)$ is the unique submodule of
$\overline{U}$ that is isomorphic to $V$, and $\pi$ induces a $kG$-module isomorphism
$\varphi:\overline{U}/\iota(V)\to V$. Since $\mathrm{Ext}^1_{kG}(\overline{U},V)=0$ and since 
the kernel of every surjective $kG$-module homomorphism $\overline{U}\to V$ is equal to
$\iota(V)$, we can argue as in the proof of \cite[Lemma 2.5]{3quat} to show that $R(G,V)/2R(G,V)$ is 
isomorphic to $k[t]/(t^2)$ and that the universal mod 2 deformation of $V$ is given by the isomorphism 
class of $\overline{U}$. Using the decomposition matrices provided in \cite[Appendix]{brauerpaper}
together with \cite[Prop. (23.7)]{CR},
we see that $V$ always has at least one lift over $W$. Therefore, it follows by \cite[Lem. 2.1]{bc5}
that $R(G,V)\cong W[[t]]/(t(t-2\mu),a2^mt)$ for certain $\mu\in W$, $a\in\{0,1\}$ and $m\in\mathbb{Z}^+$
depending on $V$.

In the case when $Q\in\{2\mathcal{A},2\mathcal{B}\}$, the decomposition
matrix of $B$ together with \cite[Prop. (23.7)]{CR} show that $V$ has 2 non-isomorphic 
lifts over $W$, which implies that $\mu\neq 0$ and $a=0$. In other words, 
$R(G,V)\cong W[[t]]/(t^2-2\mu \,t)$ for some non-zero $\mu\in W$.

On the other hand, if $Q\not\in\{2\mathcal{A},2\mathcal{B}\}$ then the defect groups of $B$
must be semidihedral. Moreover, $\overline{U}$ lies at the end of a 3-tube and 
the stable endomorphism ring of $\overline{U}$ is isomorphic to $k$. 
If $a=0$ then $R(G,V)\cong  W[[t]]/(t(t-2\mu))$ is free over $W$. If $a=1$ then
$R(G,V)/2^mR(G,V)\cong (W/2^mW)[[t]]/(t(t- 2\mu))$ is free over $W/2^mW$. 
Therefore it follows that if $a = 0$ (resp. $a = 1$), then there is a lift of $\overline{U}$, 
when regarded as a $kG$-module, over $W$ (resp. $W/2^mW$). However, arguing similarly as in 
the proof of \cite[Prop. 6.5]{brauerpaper}, we see that $R(G,\overline{U})\cong k$, which means 
we must have $a = 1$ and $m = 1$. This proves part (b) of Theorem \ref{thm:maindetail}.

To prove part (d), suppose that $V\in\mathcal{E}_4$. Since $\mathrm{Ext}^1_{kG}(V,V)=0$, 
$R(G,V)$ is isomorphic to a quotient algebra of $W$. Since $V$ is of length at most 4 and
has either a simple radical quotient or a simple socle, we can use the decomposition matrix of 
$B$ provided in \cite[Appendix]{brauerpaper} together with \cite[Prop. (23.7)]{CR} to see
that $V$ has a lift over $W$. This implies $R(G,V)\cong W$.

The last statement of the theorem is obvious for parts (c) and (d). For part (a), this follows from
\cite[Lem. 5.5]{brauerpaper}. For part (b), this follows since $W[[t]]/(t^2-2\mu \,t)$ 
for non-zero $\mu\in W$ (resp. $W[[t]]/(t^2,2t)$)
is isomorphic to a subalgebra (resp. quotient algebra) of $W[\mathbb{Z}/2]\cong W[[t]]/(t^2-2t)$.
This completes the proof of Theorem \ref{thm:maindetail}.
\end{proof}


\end{document}